\documentclass{article}

\usepackage{amsmath}
\usepackage{amssymb,amsbsy,float}
\usepackage{graphicx}
\usepackage{amsfonts}
\usepackage{latexsym}
\usepackage{hyperref}

\newtheorem{theorem}{Theorem}[section]

\newtheorem{claim}[theorem]{Claim}

\newtheorem{lemma}[theorem]{Lemma}

\newtheorem{proposition}[theorem]{Proposition}

\newenvironment{proof}[1][Proof]{\textbf{#1.} }{\ \rule{0.5em}{0.5em} \vspace{0.1 cm}}
\newcommand{\finpreuve}{\ \rule{0.5em}{0.5em} \vspace{0.1 cm}}

\newcommand{\bx}{\mbox{$\textbf{x}$}}
\newcommand{\bU}{\mbox{$\textbf{U}$}}

\newcommand{\bb}{\mbox{$\textbf{b}$}}
\newcommand{\by}{\mbox{$\textbf{y}$}}%
\newcommand{\be}{\mbox{$\textbf{e}$}}%
\newcommand{\bC}{\mbox{$\textbf{C}$}}%
\newcommand{\bp}{\mbox{$\textbf{p}$}}%
\newcommand{\bq}{\mbox{$\textbf{q}$}}
\newcommand{\bn}{\mbox{$\textbf{n}$}}%
\newcommand{\bz}{\mbox{$\textbf{z}$}}
\newcommand{\boldf}{\mbox{$\textbf{f}$}}
\newcommand{\R}{\mbox{$\mathbb{R}$}}%

\begin{document}
\title{Evolutionary dynamics may eliminate all strategies used in correlated equilibrium}
\author{Yannick Viossat\thanks{E-mail address: yannick.viossat@polytechnique.org. Address:
CEREMADE, Universit\'{e} Paris-Dauphine, Place du Mar\'{e}chal de Lattre de Tassigny, 75016 Paris}
\thanks{This article is a short version of a working paper with the same name,
published at the Stockholm School of Economics. It is based on chapter 10 of my Ph.D. dissertation, 
written at the Laboratoire d'\'{e}conom\'{e}trie de l'Ecole polytechnique under the supervision of Sylvain Sorin. 
I am much grateful to him, J\ ¨orgen Weibull, and Larry Samuelson. I also thank two anonymous referees, the editor, 
and seminar audiences at the Maison des Sciences Economiques (Universit\'{e} Paris 1), the Institut Henri Poincar\'{e}, 
the Stockholm School of Economics, Tel-Aviv University, the Technion and the Hebrew University of Jerusalem. All errors and shortcomings are mine.}}
\date{CEREMADE, Universit\'{e} Paris-Dauphine}
\maketitle
\begin{abstract}
We show on a $4 \times 4$ example that many dynamics may eliminate all strategies used in correlated equilibria, and this for an open set of games.
This holds for the best-response dynamics, the Brown-von Neumann-Nash dynamics and any monotonic or weakly
sign-preserving dynamics satisfying some standard regularity conditions. For the replicator dynamics and the
best-response dynamics, elimination of all strategies used in correlated equilibrium is shown to be robust to the
addition of mixed strategies as new pure strategies.\\


\end{abstract}

\section{Introduction}
A number of positive connections have been found between Nash equilibria and the outcome of evolutionary dynamics. For instance, for a wide clas of dynamics, if a solution converges to a point from an interior initial condition, then this point is a Nash equilibrium (Weibull, 1995). However, solutions of evolutionary dynamics need not converge and may cycle away from the set of Nash equilibria (Zeeman, 1980;  Hofbauer and Sigmund, 1998).

Since the set of correlated equilibria of a game is often much larger than its set of Nash equilibria, it might be hoped that correlated equilibria better capture the outcome of evolutionary dynamics than Nash equilibria.  
This hope is reinforced by the recent litterature on adaptive processes converging, in a time-average sense, to the set of correlated equilibria (Hart, 2005).

It was found, however, that there are games for which, for some initial conditions, the replicator dynamics eliminate all strategies belonging to the support of at least one correlated equilibrium (Viossat, 2007a). Thus, only strategies that do not take part in any equilibrium remain, rulling out convergence of any kind of time-average to the set of correlated equilibria.

The purpose of this article is to show, on a $4 \times 4$ example, that elimination of all strategies used in correlated equilibrium does not only occur under the replicator dynamics and for very specific games, but for many dynamics and for an open set of games. We also study the robustness of this result when agents are explicitly allowed to use mixed strategies.

The article is organized as follows. After presenting the framework and notations, we introduce the games we consider and explain the technique used to show that all strategies used in correlated equilibrium are eliminated (section \ref{sec:games}). Sections \ref{sec:posmon}, \ref{sec:BR} and \ref{sec:BNN} deal in turn with monotonic or weakly sign-preserving dynamics, the best-response dynamics and the Brown-von Neumann-Nash dynamics. Section \ref{sec:mixed-as-new-pure} and the appendix show that 
elimination of all strategies used in correlated equilibrium still occurs when agents are explicitly allowed to play mixed strategies. Section \ref{sec:discussion} concludes.

\textbf{Framework and notations. } We study single-population
dynamics in two-player, finite symmetric games. The set of pure
strategies is $I=\{1,2,..,N\}$ and $S_N$ denotes the simplex of
mixed strategies (henceforth, ``the simplex"). Its vertices $\be_i$,
$1 \leq i \leq N$, correspond to the pure strategies of the game. We
denote by $x_i(t)$ the proportion of the population playing strategy
$i$ at time $t$ and by $\bx(t)=(x_1(t),...,x_N(t)) \in S_N$ the
population profile (or mean strategy). We study its evolution under
dynamics of type $\dot{\bx}(t)=f(\bx(t), \bU)$, where
$\bU=(u_{ij})_{1 \leq i,j \leq N}$ is the payoff matrix of the game.
We often skip the indication of time. For every $\bx$ in $S_N$, the
probability distribution on $I \times I$ induced by $\bx$ is denoted
by $\bx \otimes \bx$. If $A$ is a subset of $S_N$, then $conv(A)$ denotes its convex hull.

We assume known the definition of a correlated equilibrium distribution (Aumann, 1974) and, with a slight
abuse of vocabulary, we write throughout correlated equilibrium for correlated equilibrium distribution. A
pure strategy $i$ is {\em used in correlated equilibrium} if there exists a correlated equilibrium $\mu$
under which strategy $i$ has positive marginal probability (since the game is symmetric, whether we restrict
attention to symmetric correlated equilibria or not is irrelevant; see footnote 2 in
(Viossat, 2007a)).
Finally, the pure strategy $i$ is \emph{eliminated} (for a given solution $\bx(\cdot)$ of a given dynamics)
if $x_{i}(t) \to 0$ as $t \to + \infty$.

\section{A family of games with a unique correlated equilibrium}
\label{sec:games}

The games considered in (Viossat, 2007a) were $4 \times 4$ symmetric games with payoff matrix

\begin{equation}
\label{eq:RSPA}
\bU_{\alpha}= \left(
\begin{array}{ccc|c}
0 & -1   & \varepsilon & -\alpha \\
\varepsilon &  0  & -1 & -\alpha \\
-1 & \varepsilon  &  0 & -\alpha \\
\hline
\frac{-1 + \varepsilon}{3} + \alpha & \frac{-1 + \varepsilon}{3} + \alpha & \frac{-1 + \varepsilon}{3} + \alpha & 0
\end{array}
\right)
\end{equation}
with $\varepsilon$ in $]0,1[$, and $0<\alpha<(1-\varepsilon)/3$. The $3 \times 3$ game obtained by omitting the fourth strategy is a Rock-Paper-Scissors game (RPS). This game has a unique Nash equilibrium : $(1/3,1/3,1/3)$, which is also the unique correlated equilibrium. When $\alpha=0$, the fourth strategy of the full game earns the same payoff as $\bn=(1/3,1/3,1/3,0)$, and there is a segment of symmetric Nash equilibria : for every $\bx \in [\bn,\be_4]=\{\lambda \bn + (1-\lambda) \be_4, \lambda \in [0,1]\}$, $(\bx,\bx)$ is a Nash equilibrium. For $\alpha>0$, $\be_4$ earns more than $\bn$, so $(\be_4, \be_4)$ is a strict Nash equilibrium, and the unique correlated equilibrium is $\be_4 \otimes \be_4$. However, for $\alpha$ small enough, the best-response cycle $\be_1 \to \be_2 \to \be_3 \to \be_1$ remains and the corresponding set :
\begin{equation}
\label{44a-eq:def_Gamma}
\Gamma=\{\bx \in S_4 : x_4=0 \mbox{ and } x_1x_2x_3=0\}.
\end{equation}
is asymptotically stable under the replicator dynamics
\begin{equation*}
\dot{x}_{i}(t)=x_{i}(t)\left[ (\bU\bx(t))_{i}-
\bx(t)\cdot\bU\bx(t)\right].
\end{equation*}
It follows that there exist games for which, for an open set of initial conditions, the replicator dynamics eliminate all strategies used in correlated equilibrium (Viossat, 2007a).

This article shows that elimination of all strategies used in correlated equilibrium does not only occur for non-generic games and the replicator dynamics, but for an open set of games and many other  dynamics. This is done by showing that, for many dynamics, there are values of $\alpha$ and $\varepsilon$ such that, for every game in a neighborhood of (\ref{eq:RSPA}):

(i) the unique correlated equilibrium is $\be_4 \otimes \be_4$;

(ii) for an open set of initial conditions, strategy $4$ is eliminated.\\
\noindent Point (i) is the object of the following proposition:
\begin{proposition}
\label{prop:uniquecor}
For every $\varepsilon$ in $]0,1[$ and every $\alpha$ in $]0,(1-\varepsilon)/3[$, every game in the neighborhood of (\ref{eq:RSPA}) has a unique correlated equilibrium: $\be_4 \otimes \be_4$.
\end{proposition}
\begin{proof}
Since the set of games with a unique correlated equilibrium is open (Viossat, 2007b) and game 
(\ref{eq:RSPA}) has a unique correlated equilibrium, it follows that every game in a neighborhood of (\ref{eq:RSPA}) has a unique correlated equilibrium. Since $\be_4 \otimes \be_4$ is clearly a correlated equilibrium of every game sufficiently close to (\ref{eq:RSPA}), the result follows.
\end{proof}

To prove (ii), a first method is to show that in (\ref{eq:RSPA}), and every nearby game, the cyclic attractor of the underlying RPS game is still asymptotically stable. This is the method we use for monotonic dynamics and for weakly sign-preserving dynamics. When in the underlying RPS game the attractor is not precisely known, but the Nash equilibrium is repelling, another method may be used. It consists in showing that there is a tube surrounding the segment $[\bn,\be_4]$ which repels solutions and such that outside of this tube, $x_4$ decreases along all trajectories. We use this method for the Brown-von-Neumann-Nash dynamics. For the best-response dynamics, both methods work. 

\section{Monotonic or weakly sign-preserving dynamics}
\label{sec:posmon}
We first need some definitions. Consider a dynamics of the form
\begin{equation}
\label{eq:gendyn}
\dot{x}_i=x_i g_i(\bx)
\end{equation}
where the $C^{1}$ functions $g_i$ have the property that $\sum_{i
\in I} x_i g_i(\bx)=0$ for all $\bx$ in $S_4$, so that the simplex
$S_4$ and its boundary faces are invariant. Such a dynamics is {\em monotonic} if the growth rates of the
different strategies are ranked according to their
payoffs:\footnote{This property goes under various names in the
literature: {\em relative monotonicity} in (Nachbar, 1990), {\em
order-compatibility} of pre-dynamics in (Friedman, 1991), {\em
monotonicity} in (Samuelson and Zhang, 1992), which we follow, and
{\em payoff monotonicity} in (Hofbauer and Weibull, 1996).}
\begin{equation*}  
g_i(\bx) > g_j(\bx) \Leftrightarrow (\bU\bx)_i > (\bU\bx)_j
\hspace{0.5cm} \forall i \in I,\forall j \in I.
\end{equation*}
It is {\em weakly sign-preserving} (WSP)  (Ritzberger and Weibull, 1995) if whenever
a strategy earns below average, its growth rate is negative:
\begin{equation*}
\left[(\bU\bx)_i < \bx \cdot \bU\bx\right] \Rightarrow g_i(\bx)<0.
\end{equation*}
Dynamics\footnote{Instead of dynamics of type (\ref{eq:gendyn}),
Ritzberger and Weibull (1995) consider dynamics of the more general
type $\dot{x}_i=h_i(\bx)$, that need not leave the faces of the
simplex positively invariant. Thus, we only consider a subclass of
their WSP dynamics.} of type (\ref{eq:gendyn}) implicitly
depend on the payoff matrix $\bU$. Thus, a more correct writing of
(\ref{eq:gendyn}) would be: $\dot{x}_i=x_i g_i(\bx,\bU)$. 
Such a dynamics \textit{depends continuously on the payoff matrix} if, for every $i$
in $I$, $g_i$ depends continuously on $\bU$. A prime example of a
dynamics of type (\ref{eq:gendyn}) which is monotonic, WSP, and
depends continuously on the payoff matrix is the replicator
dynamics.

Finally, a closed subset $C$ of $S_4$ is \emph{asymptotically stable} if it is both:

(a) Lyapunov stable:  for every neighborhood $N_1$ of $C$, there
exists a neighborhood $N_2$ of $C$ such that, for every initial
condition $\bx(0)$ in $N_2$, $\bx(t) \in N_1$ for all $t \geq 0$.

(b) locally attracting: there exists a neighborhood $N$ of $C$ such
that, for every initial condition $\bx(0)$ in $N$, $\min_{c \in C} ||\bx(t)-c|| \to_{t \to +\infty} 0$
(where $||\cdot||$ is any norm on $\R^I$).
\begin{proposition} \label{prop:genREP}
Fix a monotonic or WSP dynamics (\ref{eq:gendyn}) that depends continuously on the payoff matrix. For every $\alpha$ in $]0,1/3[$, there exists $\varepsilon>0$ such that
for every game in the neighborhood of (\ref{eq:RSPA}), the set $\Gamma$ defined by (\ref{44a-eq:def_Gamma}) is asymptotically
stable.
\end{proposition}

\noindent \textbf{Proof for monotonic dynamics}.
Consider a monotonic dynamics (\ref{eq:gendyn}). Under this dynamics, for every game in the
neighborhood of (\ref{eq:RSPA}), the set $\Gamma$ is a heteroclinic cycle. That is, a set consisting of
saddle rest points and the saddle orbits connecting these rest points. Thus we may use the asymptotic
stability's criteria for heteroclinic cycles developed by
\href{http://homepage.univie.ac.at/Josef.Hofbauer/hetcy.htm}{Hofbauer (1994)} (a more accessible reference
for this result is theorem 17.5.1 in (Hofbauer and Sigmund, 1998)). Specifically, associate with the
heteroclinic cycle $\Gamma$ its so-called characteristic matrix. That is, the $3 \times 4$ matrix whose entry
in row $i$ and column $j$ is $g_j(\be_i)$ (for $i \neq j$, this is the eigenvalue in the direction of $\be_j$
of the linearization of the vector field at $\be_i$):
$$
\begin{array}{c|cccc}
& 1 & 2 & 3 & 4 \\ \hline 
\be_1 & 0 & g_2(\be_1) & g_3(\be_1) & g_4(\be_1)\\
\be_2 & g_1(\be_2) & 0 & g_3(\be_2) & g_4(\be_2) \\
\be_3 & g_1(\be_3) & g_2(\be_3) & 0 & g_4(\be_3)
\end{array}
$$
($g_i(\be_i)=0$ because $\be_i$ is a rest point of
(\ref{eq:gendyn})).

Call $\bC$ this matrix. If $\bp$ is a real vector, let $\bp<0$ (resp. $\bp>0$) mean that all coordinates of
$\bp$ are negative (resp. positive). \href{http://homepage.univie.ac.at/Josef.Hofbauer/hetcy.htm}{Hofbauer
(1994)} shows that if the following conditions are satisfied, then $\Gamma$ is asymptotically stable:
\begin{equation}
\label{cond2} \mbox{There exists a vector $\bp$ in $\mathbb{R}^{4}$
such that $\bp >0$ and $\bC \bp<0$. } \hfill
\end{equation}
\begin{equation}
\label{cond1} \mbox{$\Gamma$ is asymptotically stable within the
boundary of $S_4$. }\footnote{That is, for each proper face
(subsimplex) $F$ of $S_4$, if $\Gamma \cap F$ is nonempty, then
it is asymptotically stable for the dynamics restricted to $F$.}
\end{equation}
Therefore, to prove proposition \ref{prop:genREP}, it is enough to show that for every
$\alpha$ in $]0,1/3[$, there exists $\varepsilon>0$ such that, for
every game in the neighborhood of (\ref{eq:RSPA}), conditions
(\ref{cond2}) and (\ref{cond1}) are satisfied. We begin with a
lemma. In the remainder of this section, $i \in \{1,2,3\}$ and $i-1$ and $i+1$ are counted modulo $3$.
\begin{lemma}
\label{lm:paymon}
For every $0 < \alpha < 1/3$, there exists $\varepsilon>0$ such that in game (\ref{eq:RSPA}), for every $i$ in $\{1,2,3\}$,
\begin{equation}
\label{eq:payoffmon-lemma}
g_4(\be_i)<0 \, \mbox{ and } \, 0< g_{i+1}(\be_i)<-g_{i-1}(\be_i).
\end{equation}
\end{lemma}
\begin{proof}  
For $\varepsilon > 0$, at the vertex $\be_i$, the payoff of strategy
$4$ (resp. $i+1$) is strictly smaller (greater) than the payoff of
strategy $i$. Since the growth rate of strategy $i$ at $\be_i$ is
$0$, this implies by monotonicity $g_4(\be_i)<0$ (resp.
$g_{i+1}(\be_i)>0$). It remains to show that
$g_{i+1}(\be_i)<-g_{i-1}(\be_i)$. For $\varepsilon=0$, we have:
$(\bU\be_i)_i=(\bU\be_i)_{i+1}>(\bU\be_i)_{i-1}$ so that
$0=g_{i+1}(\be_i)>g_{i-1}(\be_i)$. Therefore
$g_{i+1}(\be_i)<-g_{i-1}(\be_i)$ and since the dynamics depends
continuously on the payoff matrix, this still holds for small positive $\varepsilon$.
\end{proof}\\
We now prove proposition \ref{prop:genREP}.  Fix $\alpha$ and
$\varepsilon$ as in lemma \ref{lm:paymon}. Note that since the dynamics
we consider depends continuously on the payoff matrix, there exists
a neighborhood of the game (\ref{eq:RSPA}) in which the strict
inequalities (\ref{eq:payoffmon-lemma}) still hold. Thus, to prove
proposition \ref{prop:genREP}, it suffices to show that
(\ref{eq:payoffmon-lemma}) implies (\ref{cond2}) and (\ref{cond1}).\\

\noindent \textbf{(\ref{eq:payoffmon-lemma})  $\Rightarrow$ (\ref{cond2}) : } It
follows from (\ref{eq:payoffmon-lemma}) that if $p_1=p_2=p_3=1$ and $p_4>0$, then $\bC \bp<0$. 
Therefore, condition (\ref{cond2}) is satisfied.\\

\noindent \textbf{(\ref{eq:payoffmon-lemma}) $\Rightarrow$  (\ref{cond1}) : } 
We use again characteristic matrices. Let $\hat{\bC}$ denote the $3
\times 3$ matrix obtained from $\bC$ by omitting the fourth column. This corresponds to the characteristic
matrix of $\Gamma$, when viewed as a heteroclinic cycle of the underlying $3 \times 3$ RPS game. In this RPS
game, the set $\Gamma$ is trivially asymptotically stable on the relative boundary of $S_3$ ($\Gamma$
\emph{is} the relative boundary!). Furthermore, for $\hat{\bp}=(1/3,1/3,1/3)>0$, the last inequality in
(\ref{eq:payoffmon-lemma}) implies that $\hat{\bC}\hat{\bp}<0$. Therefore, it follows from theorem 1 of
\href{http://homepage.univie.ac.at/Josef.Hofbauer/hetcy.htm}{Hofbauer (1994)} that, in the $4 \times 4$
initial game, $\Gamma$ is asymptotically stable on the face spanned by $\be_1,\be_2,\be_3$. Asymptotic
stability on the face spanned by $\be_i,\be_{i+1},\be_4$ follows easily from the following facts :
on this face, $\be_{i+1}$ is a sink, $\be_i$ a saddle, every solution starting in $]\be_{i},\be_{i+1}] $ converges to $\be_{i+1}$, and $\bx(t)$ depends smoothly on $\bx(0)$. This concludes the proof.{\ \rule{0.5em}{0.5em} \vspace{0.1 cm}}\\

\noindent \textbf{Proof of proposition \ref{prop:genREP} for WSP dynamics}. 
The proof is exactly the same, except for the proof of lemma \ref{lm:paymon}, which is as follows:  Fix a WSP dynamics (\ref{eq:gendyn}).
For concreteness, set $i=2$. At $\be_2$, strategy $4$ earns less
than average. Therefore $g_4(\be_2)<0$. Now consider the case
$\varepsilon=0$: at every point $\bx$ in the relative interior of the
edge $[\be_1, \be_2]$, strategy $3$ earns strictly less than average
hence its growth rate is negative. By continuity at $\be_2$ this
implies $g_3(\be_2)\leq 0$. Since at $\be_2$, strategy $1$ earns
strictly less than average, it follows that $g_1(\be_2)<0$, hence
$g_3(\be_2)<-g_1(\be_2)$. Since the dynamics depends continuously on
the payoff matrix, this still holds for small positive $\varepsilon$.

To establish (\ref{eq:payoffmon-lemma}), it suffices to show that
$g_3(\be_2)$ is positive for every sufficiently small positive
$\varepsilon$. Let $\varepsilon>0$. If $\lambda>0$ is sufficiently small
then, for all $\mu>0$ small enough, the unique strategy which earns
weakly above average at $\bx=(\lambda\mu,1-\mu -\lambda\mu,\mu,0)$
is strategy $3$, hence $g_i(\bx) < 0$ for $i \neq 3$. Since $\sum_{1
\leq i \leq 4} x_ig_i(\bx)=0$, it follows that $x_1g_1(\bx)
+x_3g_3(\bx)>0$, hence $\lambda\mu g_1(\bx) + \mu g_3(\bx) > 0$,
hence $g_3(\bx)>- \lambda g_1(\bx)$. Letting $\mu$ go to zero, we
obtain $g_3(\be_2)\geq - \lambda g_1(\be_2)>0$ ($g_1(\be_2)<0$ was proved in the previous paragraph). {\ \rule{0.5em}{0.5em} \vspace{0.1 cm}}
\section{Best-response dynamics} \label{sec:BR}
\subsection{Main result} The best-response dynamics (Gilboa and Matsui, 1991; Matsui, 1992) is
given by the differential inclusion:
\begin{equation}
\label{eq:defBR}
\dot{\bx}(t) \in BR(\bx(t))-\bx(t),
\end{equation}
where $BR(\bx)$ is the set of best responses to $\bx$:
$$BR(\bx)=\{\mbox{$\textbf{y}$} \in S_N: \by \cdot \bU\bx
=\max_{\bz \in S_N} \bz \cdot \bU \bx\}.$$
A solution $\bx(\cdot)$ of the best-response dynamics is an
absolutely continuous function satisfying (\ref{eq:defBR}) for
almost every $t$. For the games and the initial conditions that we will consider, there is a unique
solution starting from each initial condition.\footnote{We focus on
forward time and never study whether a solution is uniquely
defined in backward time.}

Consider a $4 \times 4$ symmetric game with payoff matrix $\bU$. Let
\begin{equation}\label{eq:BR-def-V-and-W}
V(\bx):=\max_{1 \leq i \leq 3}\left[(\bU\bx)_i-\sum_{1 \leq i \leq
4} u_{ii} x_i\right] \mbox{ and } W(\bx):=\max(x_4,|V(\bx)|).
\end{equation}
For every game sufficiently close to (\ref{eq:RSPA}), the set
\begin{equation}
\label{eq:def-ST}
ST:=\{\bx \in S_4: W(\bx)=0\} 
\end{equation}
is a triangle, which, following Gaunersdorfer and Hofbauer (1995),
we call the Shapley triangle. 
\begin{proposition}
\label{prop:convBR}
For every game sufficiently close to (\ref{eq:RSPA}), if strategy $4$ is not a best response to
$\bx(0)$, then for all $t \geq 0$, $\bx(t)$ is uniquely
defined, and $\bx(t)$ converges to the Shapley triangle
(\ref{eq:def-ST}) as $t \to +\infty$. 
\end{proposition}
\begin{proof}
We begin with a lemma, which is the continuous time version of the improvement principle of
\href{http://www.sfb504.uni-mannheim.de/publications/dp97-12.pdf#search=\%22Fictitious\%20play\%20and\%20no\%20cycling\%20condition\%22}{
Monderer and Sela (1997)}:
\begin{lemma}[Improvement principle]
\label{lm:BR-improvement-princ} Let $t_1 < t_2$, let $\bb$ be a best
response to $\bx(t_1)$ and let $\bb' \in S_4$. Assume that
$\dot{\bx}=\bb-\bx$ (hence the solution points towards $\bb$) for
all $t$ in $]t_1,t_2[$. If $\bb'$ is a best response to $\bx(t_2)$
then $\bb' \cdot \bU \bb \geq \bb \cdot \bU \bb$, with strict
inequality if $\bb'$ is not a best response to $\bx(t_1)$.
\end{lemma}
\noindent \textbf{Proof of lemma \ref{lm:BR-improvement-princ}.}
Between $t_1$ and $t_2$, the solution points towards $b$. Therefore
there exists $\lambda$ in $]0,1[$ such that
\begin{equation}
\label{eq:proof-improvement-princ} \bx(t_2)=\lambda\bx(t_1)
+(1-\lambda)\bb. 
\end{equation}
If $\bb'$ is a best response to $\bx(t_2)$ then $(\bb'-\bb) \cdot
\bU \bx(t_2) \geq 0$ so that, substituting the right-hand-side of
(\ref{eq:proof-improvement-princ}) for $\bx(t_2)$, we get:
\begin{equation}
\label{eq:proof-improvement-princ-2}(1-\lambda) (\bb'-\bb) \cdot \bU
\bb \geq \lambda(\bb-\bb') \cdot \bU \bx(t_1).
\end{equation}
Since $\bb$ is a best response to $\bx(t_1)$, the right-hand-side of
(\ref{eq:proof-improvement-princ-2}) is nonnegative, and positive if
$\bb'$ is not a best response to $\bx(t_1)$. The result follows.
\finpreuve

\noindent \textbf{Proof of proposition \ref{prop:convBR} for game (\ref{eq:RSPA}). }
Fix a solution $\bx(\cdot)$ of (\ref{eq:defBR}) such that strategy $4$ is not a best response to $\bx(0)$. Note that for any $\bx$ in $S_4$, $BR(\bx) \neq conv\left(\{\be_1,\be_2,\be_3\}\right)$, because $\be_4$ strictly dominates $(1/3,1/3,1/3,0)$. Thus, either there is a unique best response to $\bx(0)$ or, counting $i$ modulo $3$, $BR(\bx(0))=conv\left(\{\be_i,\be_{i+1}\}\right)$ for some $i$ in $\{1,2,3\}$. Assume for concreteness that strategy $1$ is the unique best response to $\bx(0)$. The solution then initially points towards $\be_1$, until some other pure strategy becomes a best response. Due to the improvement principle (lemma \ref{lm:BR-improvement-princ}), this strategy can
only be strategy $2$. Thus, the solution must then point towards the
edge $[\be_1, \be_2]$. Since strategy $2$ strictly dominates
strategy $1$ in the game restricted to $\{1,2\} \times \{1,2\}$, 
strategy $2$ immediately becomes the unique best response. Iterating this argument, we see that the solution will point
towards $\be_2$, till $3$ becomes a best response, then towards $\be_3$, till $1$
becomes a best response again, and so on.

To show that this behaviour continues for ever, it suffices to show
that the times at which the direction of the trajectory changes do
not accumulate. This is the object of the following claim, which
will be proved in the end:
\begin{claim}
\label{cl:Brproof2} The time length between two successive times
when the direction of $\bx(t)$ changes is bounded away from zero.
\end{claim}
Now recall (\ref{eq:BR-def-V-and-W}), and note that for game (\ref{eq:RSPA}) the terms $u_{ii}$ are zero, so that $V(\bx)=\max_{1 \leq i \leq 3}(\bU\bx)_i$. Let $v(t):=V(\bx(t))$,
$w(t):=W(\bx(t))$. When $\bx(t)$ points towards $\be_i$ (with $i$ in $\{1,2,3\}$), we have $\dot{x}_4=-x_4$ and %
\begin{equation}
\label{44B-eq:LyapBR} \dot{v}=(\bU\dot{\bx})_i=(\bU(\be_i-\bx))_i = -v.
\end{equation}
Therefore $\dot{w}=-w$. Since for
almost all time $t$, $\bx(t)$ points towards $\be_1$, $\be_2$ or
$\be_3$, it follows that $w(t)$ decreases exponentially to $0$. Therefore
$\bx(t)$ converges to the Shapley triangle.

To complete the proof, we still need to prove claim \ref{cl:Brproof2}:\\

\noindent \textbf{Proof of claim \ref{cl:Brproof2}}: In what follows
$i \in \{1,2,3\}$ and $i+1$ is counted modulo $3$. Fix an initial
condition and let
$$g(t):=\max_{1 \leq i,j \leq 3}
\left[(\bU\bx(t))_i-(\bU\bx(t))_j\right]$$
denote the maximum difference between the payoffs of strategies in
$\{1,2,3\}.$ Let $t_i^k$ denote the $k^{th}$ time at which strategy $i$ becomes a
best response and choose $i$ such that $t_i^k<t_{i+1}^k$. Simple computations, detailed in (Viossat, 2006, p.11-12), show that:
\begin{equation}
\label{eq:notaccumulate} 
g(t_i^{k+1})=\frac{1}{\varepsilon^3 +
g(t_{i}^k)(1 + \varepsilon + \varepsilon^2)}
g(t_i^k).
\end{equation}
Since $\varepsilon<1$, 
it follows that for small $g(t_i^k)$, we
have $g(t^{k+1}_i)>g(t_i^k)$; therefore $g(t_i^k)$ is bounded away
from zero. Now, since $(\bU\bx(t))_i-(\bU\bx(t))_{i+1}$ decreases
from $g(t_i^k)$ to $0$ between $t_i^k$ and $t_{i+1}^k$, and since
the speed at which this quantity varies is bounded, it follows that
$t_{i+1}^k-t_i^k$ is bounded away from zero too. That is, the time
length between two successive times at which the direction of
$\bx(t)$ changes is bounded away from zero. \finpreuve

\noindent\textbf{Proof of proposition \ref{prop:convBR} for games close to (\ref{eq:RSPA}).}
Counting $i$ modulo $3$, let $\alpha_i=u_{ii}-u_{i-1,i}$ and $\beta_i=u_{i+1,i}-u_{ii}$, $i=1,2,3$. Let $i \in \{1,2,3\}$. For every game sufficiently close to $(\ref{eq:RSPA})$, $\alpha_i$ and $\beta_i$ are positive, $\alpha_1\alpha_2\alpha_3>\beta_1\beta_2\beta_3$, $u_{4i} < u_{ii}$, and strategy $4$ strictly dominates $(1/3,1/3,1/3,0)$. Furthermore, for every game satisfying these conditions, the proof of proposition \ref{prop:convBR} for game (\ref{eq:RSPA}) goes through. The only differences are that equation (\ref{44B-eq:LyapBR}) becomes
$$\dot{v}=(\bU\dot{\bx})_i- \sum_{1 \leq j \leq
4} u_{jj} \dot{x}_j=(\bU(\be_i-\bx))_i - \left(u_{ii} - \sum_{1 \leq j
\leq 4} u_{jj} x_j\right) = -v$$
and equation (\ref{eq:notaccumulate}) becomes
$$
g(t_i^{k+1})=\frac{\alpha_1\alpha_2\alpha_3}{\beta_1\beta_2\beta_3 +
g(t_{i}^k)(\alpha_1\alpha_2 +\alpha_1\beta_3+ \beta_2\beta_3)}
g(t_i^k).$$
See (Viossat, 2006) for details. This completes the proof. \end{proof}

Note that for every $\eta>0$, we may set the
parameters of (\ref{eq:RSPA}) so that the set $\{\bx \in S_4: \be_4 \in BR(\bx) \}$ has Lebesgue measure less than $\eta$. In this sense,
the basin of attraction of the Shapley triangle can be made
arbitrarily large. Similarly, for the replicator dynamics, the basin of attraction of the heteroclinic cycle (\ref{44a-eq:def_Gamma}) can be made arbitrarily large (Viossat, 2007a). For additional results on the best-response dynamics and of the replicator dynamics in $4 \times 4$ games based on a RPS game, see (Viossat, 2006).
\section{Brown-von Neumann-Nash dynamics}
\label{sec:BNN}%
The Brown-von Neumann-Nash dynamics (henceforth BNN) is given by:
\begin{equation}
\label{def:BNN}
\dot{x}_i=k_i(\bx)-x_i
\sum_{j \in I} k_j(\bx)
\end{equation}
where
\begin{equation}
\label{eq:BNN-def-ki} k_i(\bx):=\max (0,(\bU\bx)_i-\bx \cdot\bU\bx)
\end{equation}
is the excess payoff of strategy $i$ over the average payoff. 
We refer to (\href{http://homepage.univie.ac.at/Josef.Hofbauer/00sel.pdf}{Hofbauer, 2000}; Berger and Hofbauer, 2006) and
references therein for a motivation of and results on BNN.

Let $G_0$ denote the game (\ref{eq:RSPA}) with $\alpha=0$. Recall
that $\bU_0$ denotes its payoff matrix and $\bn=\left(\frac{1}{3},\frac{1}{3},\frac{1}{3},0\right)$ the mixed strategy
corresponding to the Nash equilibrium of the underlying RPS game.
It may be shown that the set of symmetric Nash equilibria of $G_0$ is the
segment $E_0=[\bn,\be_4]$.\footnote{The game $G_0$ has other, asymmetric
equilibria, but they will play no role.} This section is devoted to
a proof of the following proposition:
\begin{proposition}
\label{prop:BNN}
If $C$ is a closed subset of $S_4$ disjoint from $E_0$, then there
exists a neighborhood of $G_0$ such that, for every game in this
neighborhood and every initial condition in $C$, $x_4(t) \to 0$ as
$t \to +\infty$.
\end{proposition}
Any neighborhood of $G_0$ contains a neighborhood of a game of kind (\ref{eq:RSPA}), hence an open set of games for which $\be_4 \otimes \be_4$ is the unique correlated
equilibrium. Together with proposition
\ref{prop:BNN}, this implies that there exists an open set of games
for which, under BNN, the unique strategy played in correlated
equilibrium is eliminated from an open set of initial conditions.

The essence of the proof of proposition \ref{prop:BNN} is to show
that, for games close to $G_0$, there is a ``tube" surrounding $E_0$
such that: (i) the tube repels solutions coming from outside; (ii)
outside of the tube, strategy 4 earns less than average, hence $x_4$
decreases. We first show that in $G_0$ the segment $E_0$ is locally
repelling.

The function
$$V_0(\bx):=\frac{1}{2}\sum_{i\in I}k_i^2=\frac{1}{2}\sum_{i\in
I}\left[\max\left( 0,(\bU_0\bx)_i-\bx \cdot \bU_0
\bx\right)\right]^2$$
is continuous, nonnegative and equals $0$ exactly on the symmetric
Nash equilibria, i.e. on $E_0$, so that $V_0(\bx)$ may be seen as a
distance from $\bx$ to $E_0$. Fix an initial condition and let
$v_0(t):=V_0(\bx(t))$.
\begin{lemma}
\label{lm:BNN}
There exists an open neighborhood $N_{eq}$ of $E_0$
such that, under BNN in the game $G_0$, $\dot{v}_0(t)>0$ whenever
$\bx(t) \in N_{eq} \backslash E_0$.
\end{lemma}
\begin{proof}
It is easily checked that:
\begin{equation}
\label{eq:earn-same-bis} \bn \cdot \bU_0 \bx = \be_4 \cdot \bU_0 \bx
\hspace{0.5 cm} \forall \bx \in S_4
\end{equation}
(that is, $\bn$ and $\be_4$ always earn the same payoff) and%
\begin{equation}
\label{eq:yield-same-bis} (\bx-\bx') \cdot \bU_0 \be_4=(\bx-\bx')
\cdot \bU_0 \bn=0 \hspace{0.5 cm} \forall \bx \in S_4, \forall \bx'
\in S_4
\end{equation}
(that is, against $\be_4$ [resp. $\bn$], all strategies earn the same payoff). Furthermore, as follows from
lemma 4.1 in
(Viossat, 2007a),
for every $\bp$ in $E_0$ and every $\bx \notin E_0$,
\begin{equation}
\label{eq:RPSA-globinf-bis}%
(\bx-\bp) \cdot \bU_0 \bx = (\bx - \bp) \cdot \bU_0
(\bx-\bp)=\frac{1-\varepsilon}{2} \sum_{1 \leq i \leq 3}\left(x_i
-\frac{1-x_4}{3}\right)^2>0.
\end{equation}
\href{http://homepage.univie.ac.at/Josef.Hofbauer/00sel.pdf}{Hofbauer (2000)} shows that the function $v_0$ satisfies %
\begin{equation}
\label{eq:BNN-dotV}
\dot{v}_0=\bar{k}^2\left[(\bq-\bx)\cdot \bU_0(\bq-\bx) - (\bq-\bx) \cdot
\bU_0\bx\right]
\end{equation}
with $\bx=\bx(t)$, $\bar{k}=\sum_i k_i$ and $q_i=k_i/\bar{k}$. It
follows from equation (\ref{eq:yield-same-bis}) that if $\bp \in
E_0$, then against $\bp$ all strategies earn the same payoff.
Therefore the second term $(\bq-\bx) \cdot \bU_0\bx$ goes to $0$ as
$\bx$ approaches $E_0$. Thus, to prove lemma \ref{lm:BNN}, it
suffices to show that as $\bx$ approaches $E_0$, the first term
$(\bq-\bx) \cdot \bU_0(\bq-\bx)$ is positive and bounded away from
$0$. But for $\bx \notin E_0$,
\begin{equation}
\label{eq:BNN-comparepayoffs}
\min_{1 \leq i \leq 3}(\bU_0\bx)_i \leq \bn\cdot
\bU_0\bx=(\bU_0\bx)_4 < \bx \cdot \bU_0 \bx
\end{equation}
(the first inequality holds because $\bn$ is a convex combination of
$\be_1$, $\be_2$ and $\be_3$, the equality follows from
(\ref{eq:earn-same-bis}) and the strict inequality from
(\ref{eq:RPSA-globinf-bis}) applied to $\bp=\be_4$). It follows from
$(\bU_0\bx)_4 < \bx \cdot \bU_0 \bx$ that $k_4=0$ hence $q_4=0$;
similarly, it follows from $\min_{1 \leq i \leq 3}(\bU_0\bx)_i < \bx
\cdot \bU_0 \bx$ that $q_i=0$ for some $i$ in $\{1,2,3\}$. Together
with (\ref{eq:RPSA-globinf-bis}) applied to $\bx=\bq$, this implies that for every $\bp$
in $E_0$,
$$(\bq-\bp) \cdot \bU_0(\bq-\bp)=\frac{1-\varepsilon}{2}\sum_{1 \leq i \leq
3}\left(q_i-\frac{1}{3}\right)^2 \geq \frac{1-\varepsilon}{18}.$$
This completes the proof.%
\end{proof}

{\bf Proof of proposition \ref{prop:BNN}}. Consider first the BNN
dynamics in the game $G_0$. Recall lemma \ref{lm:BNN} and let
\begin{equation} \label{eq:BNN-def-delta}
0<\delta<\min_{\bx \in S_4\backslash N_{eq}} V_0(\bx)
\end{equation}
(the latter is positive because $V_0$ is positive on $S_4 \backslash
E_0$, hence on $S_4\backslash N_{eq}$, and because $S_4\backslash
N_{eq}$ is compact). Note that if $V_0(\bx) \leq \delta$ then $\bx
\in N_{eq}$. Therefore it follows from lemma \ref{lm:BNN} and
$\delta>0$ that
\begin{equation}
\label{eq:BNN-repell-new}%
v_0(t)=\delta \Rightarrow \dot{v}_0(t) >0.
\end{equation}
Let
$$C_{\delta}:=\{\bx \in S_4 : V_0(\bx) \geq \delta\}.$$
Since $\delta>0$, the sets $C_{\delta}$ and $E_0$ are disjoint.
Therefore, by (\ref{eq:RPSA-globinf-bis}) applied to $\bp=\be_4$,
\begin{equation}
\label{eq:BNN-4decreases-new}
\bx \in C_{\delta} \Rightarrow (\bU_0\bx)_4-\bx \cdot \bU_0\bx < 0
\end{equation}
so that $x_4$ decreases strictly as long as $\bx \in C_{\delta}$ and
$x_4>0$. Since, by (\ref{eq:BNN-repell-new}), the set $C_{\delta}$
is forward invariant, it follows that for any initial condition in
$C_{\delta}$, strategy 4 is eliminated.

Now let $\nabla V_0 (\bx)=\left(\partial V_0/\partial x_i\right)_{1
\leq i \leq n}(\bx)$ denote the gradient of $V_0$ at $\bx$. It is
easy to see that  $V_0$ is $C^1$. Therefore it follows from
(\ref{eq:BNN-repell-new}), $\dot{v}_0(t)=\nabla V_0(\bx(t)) \cdot
\dot{\bx}(t)$ and compactness of $\{\bx \in S_4: V_0(\bx)=\delta\}$
that
\begin{equation}
\label{eq:BNN-repell-old}
\exists \gamma>0, \left[v_0(t)=\delta \Rightarrow \dot{v}_0(t)\geq
\gamma>0\right].
\end{equation}
Similarly, since $C_{\delta}$ is compact, it follows from
(\ref{eq:BNN-4decreases-new}) that there exists $\gamma'>0$ such
that
\begin{equation}
\label{eq:BNN-4decreases-old}
\bx \in C_{\delta} \Rightarrow (\bU_0\bx)_4-\bx \cdot \bU_0\bx \leq -\gamma' < 0.%
\end{equation}
Since $\dot{\bx}$ is Lipschitz in the payoff matrix, it follows from
(\ref{eq:BNN-repell-old}) that for $\bU$ close enough to $\bU_0$, we
still have $v_0(t)=\delta \Rightarrow \dot{v_0}>0$ under the
perturbed dynamics. Similarly, due to (\ref{eq:BNN-4decreases-old}),
we still have $\bx \in C_{\delta} \Rightarrow (\bU\bx)_4-\bx \cdot
\bU\bx < 0$. Therefore the above reasoning applies and for every
initial condition in $C_{\delta}$, strategy 4 is eliminated.

Note that $\delta$ can be chosen arbitrarily small (see
(\ref{eq:BNN-def-delta})). Therefore, to complete the proof of
proposition \ref{prop:BNN}, it suffices to show that if $C$ is a
compact set disjoint from $E_0$ then, for $\delta$ sufficiently
small, $C \subset C_{\delta}$. But since $V_0$ is positive on $S_4
\backslash E_0$, and since $C$ is compact and disjoint from $E_0$,
it follows that there exists $\delta'>0$ such that, for all $\bx$ in $C$, $V_0(\bx)
\geq \delta'$; hence, for all $\delta \leq \delta'$, $C \subset
C_{\delta}$. This completes the proof. \finpreuve

\href{http://homepage.univie.ac.at/Josef.Hofbauer/00sel.pdf}{Hofbauer (2000, section 6)} considers the
following generalization of the BNN dynamics:
\begin{equation}
\label{eq:generalBNN} \dot{x}_i= f(k_i)-x_i \sum_{j=1}^{n} f(k_j)
\end{equation}
where $f: \R_+ \to \R_+$ is a continuous function with $f(0)=0$ and
$f(u)>0$ for $u>0$, and where $k_i$ is defined as in
(\ref{eq:BNN-def-ki}). The results of this section generalize to any
such dynamics:
\begin{proposition} \label{prop:generalBNN} Consider a dynamics of type (\ref{eq:generalBNN}). If $C$ is a
closed subset of $S_4$ disjoint from $E_0$, then there exists a
neighborhood of $G_0$ such that, for every game in this neighborhood
and every initial condition in $C$, $x_4(t) \to 0$ as $t \to
+\infty$.
\end{proposition}
\begin{proof} %The proof is the same as for BNN.
Replace $V_0(\bx)$ by $W_0(\bx):=\sum_i F(k_i(\bx))$, where $F$ is an anti-derivative of $f$, and replace
$k_i$ by $f(k_i)$. Let $\bar{f}=\sum_i f(k_i)$, $\tilde{f}_i=f(k_i)/\bar{f}$, and
$\tilde{\boldf}=\left(\tilde{f}_i\right)_{1 \leq i \leq N}$. Finally, let $w_0(t)=W_0(\bx(t))$. As shown by
\href{http://homepage.univie.ac.at/Josef.Hofbauer/00sel.pdf}{Hofbauer (2000)},
$$\dot{w}_0= \bar{f}^{\,2}\left[(\tilde{\boldf}-\bx)\cdot \bU_0(\tilde{\boldf}-\bx) - (\tilde{\boldf}-\bx) \cdot
\bU_0\bx\right]$$
which is the analogue of (\ref{eq:BNN-dotV}). Then apply exactly the
same proof as for BNN.
\end{proof}

%%%%%%%%%%%%%%%%%%%%%%%%%%
%%%%%%%%%%%%%%%%%%%%%%%%%%%%%%%%%%%%%%%%%%%%%%%%%%%%%%%%%%%%%%%%%%%%%%%%%%%%%%%%%%%%%%%%%%%%%%%%%%%
%%%%%%%%%%%%%%%%%%%%%%%%%%%%%%%%%%%%%%%%%%%%%%%%
%%%%%%%%%%%%%%%%%%%%%%%%%%
%%%%%%%%%%%%%%%%%%%%%%%%%%%%%%%%%%%%%%%%%%%%%%%%%%%%%%%%%%%%%%%%%%%%%%%%%%%%%%%%%%%%%%%%%%%%%%%%%%%
%%%%%%%%%%%%%%%%%%%%%%%%%%%%%%%%%%%%%%%%%%%%%%%%%%%%%%%%%%%%%%%%%%%%%%%%%%%%%%%%%%%%%%%%%%%%%%%%%%%%%%%%%%%%%%%%%%%%%%
%%%%%%%%%%%%%%%%%%%%%%%%%%%%%%%%%%%%%%%%%%%%%%%%%%%%%%%%%%%%%%%%%%%%%%%%%%%%%%%%%%%%%%%%%%%%%%%%%%%%%%%%%%%%%%%%%%%%%%
%%%%%%%%%%%%%%%%%%%%%%%%%%%%%%%%%%%%%%%%%%%%%%%%%%%%%%%%%%%%%%%%%%%%%%%%%%%%%%%%%%%%%%%%%%%%%%%%%%%%%%%%%%%%%%%%%%%%%%
%%%%%%%%%%%%%%%%%%%%%%%%%%%%%%%%%%%%%%%%%%%%%%%%%%%%%%%%%%%%%%%%%%%%%%%%%%%%%%%%%%%%%%%%%%%%%%%%%%%%%%%%%%%%%%%%%%%%%%
\section{Robustness to the addition of mixed strategies as new pure strategies}

\label{sec:mixed-as-new-pure} We showed that for many dynamics,
there exists an open set of symmetric $4 \times 4$ games for which,
from an open set of initial conditions, the unique strategy used in
correlated equilibrium is eliminated. Since we might not want to
rule out the possibility that individuals use mixed strategies, and
that mixed strategies be heritable, it is important to check whether
our results change if we explicitly introduce mixed strategies as
new pure strategies of the game. The paradigm is the following (Hofbauer and Sigmund, 1998, section 7.2): there is an underlying normal-form
game, called the base game, and a finite number of types of agents.
Each type plays a pure or mixed strategy of the base game. We assume
that each pure strategy of the base game is played (as a pure
strategy) by at least one type of agent, but otherwise we make no
assumptions on the agents' types. The question is whether we can
nonetheless be sure that, for an open set of initial conditions, all
strategies used in correlated equilibrium are eliminated. This
section shows that the answer is positive, at least for the
best-response dynamics and the replicator dynamics. We first
need some notations and vocabulary.

Let $G$ be a finite game with strategy set $I=\{1,...,N\}$ and
payoff matrix $\bU$. A finite game $G'$ is \emph{built on $G$ by
adding mixed strategies as new pure strategies} if:

First, letting $I'=\{1,...,N,N+1,...,N'\}$ be the set of pure
strategies of $G'$ and $\bU'$ its payoff matrix, we may associate to
each pure strategy $i$ in $I'$ a mixed strategy $\bp^i$ in $S_N$ in
such a way that:
\begin{equation} \label{eq:new-pure-are-old-mixed}
\forall i \in I',\forall j \in I', \be'_i \cdot \bU' \be'_j=\bp^i
\cdot \bU \bp^j
\end{equation}
where $\be'_i$ is the unit vector in $S_{N'}$ corresponding to the
pure strategy $i$.

Second, if $1 \leq i \leq N$, the pure strategy $i$ in the game $G'$
corresponds to the pure strategy $i$ in the base game $G$:
\begin{equation}
\label{eq:oldstrat-available} 1 \leq i \leq N \Rightarrow
\bp^i=\be_i.
\end{equation}
If $\mu'=(\mu(k,l))_{1 \leq k,l \leq N'}$ is a probability
distribution over $I' \times I'$, then it induces the probability
distribution $\mu$ on $I \times I$ given by:
$$\mu(i,j)=\sum_{1 \leq k,l \leq N'} \mu'(k,l) p^{k}_{i} p^{l}_{j}
\hspace{1 cm} \forall (i,j) \in I \times I.$$
It follows from a version of the revelation principle (see Myerson,
1994) that, if $G'$ is built on $G$ by adding mixed strategies as
new pure strategies, then for any correlated equilibrium $\mu'$ of
$G'$, the induced probability distribution on $I \times I$ is a
correlated equilibrium of $G$. Thus, if $G$ is a $4 \times 4$
symmetric game with $\be_4 \otimes \be_4$ as unique correlated
equilibrium, then $\mu'$ is a correlated equilibrium of $G'$ if and
only if, for every $k,l$ in $I'$ such that $\mu'(k,l)$ is positive,
$\bp^k=\bp^l=\be_4$. Thus, the unique strategy of $G$ used in
correlated equilibria of $G'$ is strategy $4$. We show below that:
\begin{proposition}\label{prop:mixed-as-new-pure}
For the replicator dynamics and for the best-response dynamics,
there exists an open set of $4 \times 4$ symmetric games such that,
for any game $G$ in this set:

(i) $\be_4 \otimes \be_4$ is the unique correlated equilibrium of
$G$

(ii) For any game $G'$ built on $G$ by adding mixed strategies as
new pure strategies and for an open set of initial conditions, every
pure strategy $k$ in $I'$ such that $p^k_4>0$ is
eliminated.
\end{proposition}
(the open set of initial conditions in property (ii) is a subset of
$S_{N'}$, the simplex of mixed strategies of $G'$, and may depend on
$G'$)
\subsection{Proof for the best-response dynamics}
Let $G$ be a finite game and let $G'$ be a finite game built on $G$
by adding mixed strategies of $G$ as new pure strategies. Associate to
each mixed strategy $\bx'$ in $S_{N'}$ the induced mixed strategy
$\bx$ in $S_N$ defined by:
\begin{equation}
\label{eq:extgen-induce} \bx:=\sum_{k=1}^{N'} x_{k}' \bp^k.
\end{equation}
Let $\bx'(\cdot)$ be a solution of the best-response dynamics in
$G'$ and $\bx(\cdot)$ the induced mapping from $\R_+$ to $S_N$.
\begin{proposition}
\label{prop:mixed-BR-induce} $\bx(\cdot)$ is a solution of the
best-response dynamics in $G$.
\end{proposition}
\begin{proof} For almost all $t \geq 0$, there exists a vector $\bb'
\in BR(\bx'(t))$ such that $\dot{\bx}'(t)=\bb'-\bx'(t)$. Let
$\bb:=\sum_{k \in I'} b'_k \bp^k \in S_N$. It follows from
(\ref{eq:extgen-induce}) that:
\begin{equation}
\label{eq:x-check-BR} \dot{\bx}(t)=\sum_{k = 1}^{N'} \dot{\bx}'_k
\bp^k = \sum_{k=1}^{N'} (b'_k - x'_k)\bp^k=\bb-\bx(t).
\end{equation}
Furthermore, since $\bb'$ is a best response to $\bx'(t)$, it follows
from (\ref{eq:new-pure-are-old-mixed}) and
(\ref{eq:oldstrat-available}) that $\bb$ is a best response to
$\bx(t)$. 
Together with (\ref{eq:x-check-BR}), this implies that, for almost all $t$,
$\dot{\bx} \in BR(\bx) - \bx$. The result follows.
\end{proof}

Since
$$x_i(t) \to 0 \Rightarrow  \left(\forall k \in N', \left[ p_i^k>0 \Rightarrow x'_k(t) \to 0\right] \right)$$
proposition \ref{prop:mixed-as-new-pure} follows from propositions \ref{prop:convBR} and \ref{prop:mixed-BR-induce}.
\subsection{Proof for the replicator dynamics}
Recall that $G_0$ denotes game (\ref{eq:RSPA}) with
$\alpha=0$. Since, as already mentioned, every neighborhood of $G_0$ contains an
open set of games with $\be_4 \otimes \be_4$ as unique correlated
equilibrium, it suffices to show that every game close enough to $G_0$
satisfies property (ii) of proposition \ref{prop:mixed-as-new-pure}. 
This is done in the appendix.

The intuition is the following: first note that, for a game $G$ close to $G_0$, the set
$\Gamma$ defined in (\ref{44a-eq:def_Gamma}) is an attractor, close
to which strategy $4$ earns less than average. Now consider a game
$G'$ built on $G$ by adding mixed strategies as new pure strategies and let $\Gamma'$ denote the subset of $S_{N'}$ corresponding to $\Gamma$:
$$\Gamma'=\{\bx \in S_{N'} : x_1 + x_2 +x_3=1 \mbox{ and } x_1x_2x_3=0\}$$
For an initial condition close to $\Gamma'$: (a) as long as
the share of strategies $k \geq 4$ remains low, the solution remains
close to $\Gamma'$; 
(b) as long as the solution is close to $\Gamma'$, strategy $4$ earns less than
average and its share decreases; (c) as long as the share of
strategy $4$ does not increase, the share of strategies $k \geq 5$
remains low; moreover, if the share of strategy $x_4$ decreases, so
does, on average, the share of each added mixed strategy
in which strategy $4$ is played with positive probability.

Putting (a), (b) and (c) together gives the result.
\section{Discussion}
\label{sec:discussion} We showed that elimination of all strategies used in correlated equilibrium is a
robust phenomenon, in that it occurs for many dynamics, an open set of games and an open set of initial
conditions. Furthermore, at least for some of the leading dynamics, the results are robust to the addition of
mixed strategies as new pure strategies. Under the replicator dynamics, the best-response dynamics or the Brown-von
Neumann-Nash dynamics, the basin of attraction
of the Nash equilibrium of (\ref{eq:RSPA}) can be made arbitrarily small. In particular,  the minimal distance from the cyclic attractor 
to the basin of attraction of the Nash equilibrium can be made much larger than the minimal distance from
the Nash equilibrium to the basin of attraction of the cyclic attractor. 
The latter would thus be stochastically stable in a model \`{a} la Kandori, Mailath and Rob (1993).\footnote{The
(unperturbed) dynamics used by Kandori, Mailath and Rob (1993) is a discrete-time version of the
best-response dynamics, but it could easily be replaced by a discrete-time version of another
dynamics.} 
These results show a sharp difference between evolutionary dynamics and ``adaptive heuristics" such as no-regret dynamics (Hart
and Mas-Collel, 2003; Hart, 2005) or hypothesis testing (Young, 2004, chapter 8).

Some limitations of our results should however be stressed. First,
our results have been shown here only for single-population dynamics. They imply that for \emph{some} games and \emph{some} interior
initial conditions, two-population dynamics eliminate all strategies
used in correlated equilibrium\footnote{This is because for
symmetric two-player games with symmetric initial conditions,
two-population dynamics reduce to single-population dynamics, at
least for the replicator dynamics, the best-response dynamics and
the Brown-von-Neumann Nash dynamics.}; but maybe not for an open set
of games nor for an open set of initial conditions.

Second, the monotonic and weakly sign-preserving dynamics of section
\ref{sec:posmon} are non-innovative: strategies initially absent do
not appear. This has the effect that, even when focusing on interior
initial conditions, the growth of the share of the population
playing strategy $i$ is limited by the current value of this share.
This is appropriate if we assume that agents have to meet an agent
playing strategy $i$ to become aware of the possibility of playing
strategy $i$; but in general, as discussed by e.g. Swinkels (1993,
p.459), this seems more appropriate in biology than in economics.
While our results hold also for some important innovative dynamics,
such as the best-response dynamics and a family of dynamics
including the Brown-von Neumann-Nash dynamics, more general results
would be welcome.

Third, in the games we considered, the unique correlated equilibrium
is a strict Nash equilibrium, and is thus asymptotically stable
under most reasonable dynamics, including all those we studied.
Thus, there is still an important connection between equilibria and the outcome
of evolutionary dynamics.

For Nash equilibrium, these three limitations can be overcome, at least partially: there are wide classes of
multi-population innovative dynamics for which there exists an open set of games such that, for an open set
of initial conditions, all strategies belonging to the support of at least one Nash equilibrium are
eliminated
(\href{http://ceco.polytechnique.fr/fichiers/ceco/perso/fichiers/viossat_262_Texte_these.pdf}{Viossat,
2005}, chapter 11). Moreover, for the single-population replicator dynamics or the single-population
best-response dynamics, there are games for which, for almost all initial conditions, all strategies used in
Nash equilibrium are eliminated
(\href{http://ceco.polytechnique.fr/fichiers/ceco/perso/fichiers/viossat_262_Texte_these.pdf}{Viossat,
2005}, chapter 12). Whether these results extend
to correlated equilibrium is an open question.
\begin{appendix}
\section{Proof of proposition \ref{prop:mixed-as-new-pure}  for the replicator dynamics}
We need to show that for every game close enough to $G_0$, property (ii) of proposition \ref{prop:mixed-as-new-pure} is satisfied.
As in section \ref{sec:BNN}, let $E_0=[\bn,\be_4]$, with $\bn=(1/3,1/3,1/3,0)$. 
For $\bx$ in
$S_4\backslash\{\be_4\}$, let
\begin{equation*}
V(\bx):=3\frac{(x_1x_2x_3)^{1/3}}{x_1
+x_2 +x_3}.
\end{equation*}
The function $V$ takes its maximal value $1$ on
$E_0\backslash\{\be_4\}$ and its minimal value $0$ on the set $\{\bx
\in S_4\backslash\{\be_4\} : x_1x_2x_3 = 0\}$. Fix $\delta$ in
$]0,1[$. If $V(\bx) \leq \delta$ then $\bx \notin E_0$, hence it
follows from (\ref{eq:BNN-comparepayoffs}) that, at $\bx$, strategy
4 earns strictly less than average. Together with a compactness
argument, this implies that there exists $\gamma_1>0$ such that:
\begin{equation}
\label{eq:G0-4decrease} V(\bx) \leq \delta \Rightarrow
\left[(\bU_0\bx)_4 - \bx \cdot \bU_0 \bx \leq - \gamma_1\right].
\end{equation}
Furthermore, it is shown in
(Viossat, 2007a)
that in $G_0$, under the replicator dynamics, the function $V(\bx)$ decreases strictly along interior
trajectories (except those starting in $E_0$). More precisely, for every interior initial condition $\bx(0)
\notin E_0$ and every $t$ in $\mathbb{R}$, the function $v_0(t):=V(\bx(t))$ satisfies $\dot{v}_0(t)<0$.
Together with the compactness of $\{\bx \in S_4\backslash\{\be_4\}, V(\bx)=\delta$\}, this implies that there
exists $\gamma_2>0$ such that
\begin{equation}
\label{eq:GO-vdecrease} v_0(t)=\delta \Rightarrow \dot{v}_0(t) \leq
- \gamma_2.
\end{equation}
Fix a $4 \times 4$ matrix $\bU$ and a solution $\bx(\cdot)$ of the
replicator dynamics with payoff matrix $\bU$, with $\bx(0) \neq
\be_4$. Let $v(t):=V(\bx(t))$. Thus, the difference between $v_0$
and $v$ is that, in the
definition of $v$, $\bx(\cdot)$ is a solution of the replicator dynamics for the
payoff matrix $\bU$ and not for $\bU_0$. Since $(\bU\bx)_4 - \bx
\cdot \bU \bx$ and $\dot{\bx}$ are Lipschitz in $\bU$, it follows
from (\ref{eq:G0-4decrease}) and (\ref{eq:GO-vdecrease}) that there
exists $\gamma>0$ such that, if $||\bU-\bU_0||<\gamma$:
\begin{equation}
\label{eq:mixed-4decrease} V(\bx) \leq \delta \Rightarrow
\left[(\bU\bx)_4 - \bx \cdot \bU \bx \leq - \gamma\right]
\end{equation}
and
\begin{equation}
\label{eq:mixed-vdecrease} v(t)=\delta  \Rightarrow \dot{v}(t) \leq
- \gamma.
\end{equation}
Fix a game $G$ with payoff matrix $\bU$ such that
$||\bU-\bU_0||<\gamma$. Let $G'$ be a game built on $G$ by adding
mixed strategies of $G$ as new pure strategies, and let $\bU'$ be
its payoff matrix. 
For $\bx'$ in $S_{N'}$ such that $x'_1+x'_2+x'_3>0$, let
$$V'(\bx'):=3\frac{(x'_1x'_2x'_3)^{1/3}}{x'_1 +x'_2 +x'_3}.$$
Consider a solution $\bx'(\cdot)$ of the replicator dynamics in $G'$
(with $\sum_{1 \leq i \leq 3}x'_i(0) >0$) and let
$v'(t)=V'(\bx'(t))$. On the face of $S_{N'}$ spanned by the
strategies of the original game:
$$\left\{\bx' \in S_{N'} : \sum_{1 \leq i \leq 4} x'_i=1\right\},$$
the replicator dynamics behaves just as in the base-game. Therefore,
(\ref{eq:mixed-4decrease}) and (\ref{eq:mixed-vdecrease}) imply
trivially that:
\begin{equation}
\label{eq:mixed'-4decrease} \left[\sum_{1\leq i\leq 4} x'_i = 1
\mbox{ and } V'(\bx') \leq \delta\right] \Rightarrow
\left[(\bU'\bx')_4 - \bx' \cdot \bU' \bx' \leq - \gamma\right]
\end{equation}
and
\begin{equation}
\label{eq:mixed'-vdecrease} \left[\sum_{1\leq i\leq 4} x'_i = 1
\mbox{ and } v'(t)=\delta \right] \Rightarrow \dot{v}'(t) \leq -
\gamma.
\end{equation}
Now define $\bar{\bx}' \in S_{N'}$ as the projection of $\bx'$ on the
face of $S_{N'}$ spanned by the strategies of the original game.
That is,
$$\bar{\bx}'_i=\frac{x'_i}{\sum_{1 \leq j \leq 4} x'_j} \mbox{ if } 1 \leq i
\leq 4 \mbox{, and }  \bar{\bx}'_i=0 \mbox{ otherwise.}$$
Note that $V'(\bx')=V(\bar{\bx}')$. Furthermore, a simple computation shows
that
$$\max_{1 \leq i \leq N'} |x'_i - \bar{x}'_i| \leq N' \max_{5
\leq k \leq N'} x'_k.$$
Therefore, since $(\bU'\bx')_4 - \bx' \cdot \bU' \bx'$ and the
vector field $\dot{\bx}'$ are Lipschitz in $\bx'$, it follows from
(\ref{eq:mixed'-4decrease}) and (\ref{eq:mixed'-vdecrease}) that
there exist positive constants $\eta$ and $\gamma'$ such that
\begin{equation}
\label{eq:mixed-4decrease'} \left[\max_{5 \leq k \leq N'} x'_k \leq
\eta \mbox{ and } V'(\bx') \leq \delta\right] \Rightarrow
(\bU'\bx')_4 - \bx' \cdot \bU' \bx' \leq - \gamma'
\end{equation}
and
\begin{equation}
\label{eq:mixed-vdecrease'} \left[\max_{5 \leq k \leq N'} x'_k \leq
\eta \mbox{ and } v'(t)= \delta\right] \Rightarrow \dot{v}'(t) \leq
- \gamma'.
\end{equation}

Fix $\by' \in S_{N'}$ such that
$$\sum_{1\leq i\leq 4} y'_i=1, V'(\by') < \delta \mbox{ and } C:=\min_{1 \leq i \leq 3} y'_i
>0.$$
There exists an open neighborhood $\Omega$ of $\by'$ in $S_{N'}$ such
that
$$\forall \bx' \in \Omega, \left[\, \min_{1 \leq i \leq 3} x'_i >C/2,
\hspace{0.1 cm} \max_{5 \leq k \leq N'} x'_k < C\eta/2, \, \mbox{
and } V'(\bx')<\delta \, \right].$$
Consider an interior solution $\bx'(\cdot)$ of the replicator
dynamics in $G'$ with initial condition in $\Omega$. Recall that
$\bp^k$ denotes the mixed strategy of $G$ associated with the pure
strategy $k$ of $G'$. To prove proposition
\ref{prop:mixed-as-new-pure} for the replicator dynamics, it
suffices to show that:
\begin{proposition}
For all $k$ in $\{4,...,N'\}$ such that $p^k_4>0$, $x'_k(t) \to_{t
\to + \infty} 0$.
\end{proposition}
\begin{proof}
We begin with two lemmas:
\begin{lemma}
\label{lm:leq-eta} Let $T>0$ and $k \in \{5,...,N'\}$. If $x'_4(T)
\leq x'_4(0)$ then $x'_k(T) < \eta$.
\end{lemma}
\begin{proof}
By construction of $G'$, strategy $k \in I'$ earns the same payoff
as the mixed strategy $\sum_{1 \leq i \leq 4} p^k_i \be_i'$:
$$(\bU'\bx')_k=\sum_{1 \leq i \leq 4} p^k_i (\bU'\bx')_i \hspace{1 cm} \forall \bx' \in S_{N'}.$$
Therefore, it follows from the definition of the replicator dynamics
that:
$$\frac{\dot{x}'_{k}}{x'_k}=\sum_{1 \leq i \leq 4} p_i^k \frac{\dot{x}'_i}{x'_i} \quad .$$
Integrating between $0$ and $T$ and taking the exponential of both
sides leads to:
\begin{equation}
\label{44B-eq:taking-exp} x'_k(T)=x'_k(0) \prod_{1 \leq i \leq 4}
\left(\frac{x'_i(T)}{x'_i(0)}\right)^{p_i^k}.
\end{equation}
Noting that for $1 \leq i \leq 3$, we have $x'_i(T) \leq 1$, $1/x'_i(0) \leq 2/C$ and $1 \leq 2/C$, we get:
\begin{equation}
\label{44B-eq:after-tak-exp} \prod_{1 \leq i \leq 3}
\left(\frac{x'_i(T)}{x'_i(0)}\right)^{p_i^k} \leq \prod_{1 \leq i
\leq 3}
\left(\frac{2}{C}\right)^{p_i^k}=\left(\frac{2}{C}\right)^{1-p_4^k}
\leq \frac{2}{C}.
\end{equation}
Since furthermore $x'_k(0) < C\eta/2$, we obtain from
(\ref{44B-eq:taking-exp}) and (\ref{44B-eq:after-tak-exp}):
\begin{equation}
\label{eq:leq-eta} x'_k(T) <\frac{C\eta}{2} \frac{2}{C}
\left(\frac{x_4'(T)}{x_4'(0)}\right)^{p_4^k} = \eta
\left(\frac{x_4'(T)}{x_4'(0)}\right)^{p_4^k}.
\end{equation}
The result follows.
\end{proof}
\begin{lemma}
\label{lm:leq-eta-delta} For all $t>0$, $\max_{k \in \{5,...,N'\}}
x'_k(t) < \eta$ and $v'(t)< \delta$.
\end{lemma}
\begin{proof}
Otherwise there is a first time $T>0$ such that $\max_{k \in
\{5,...,N'\}} x'_k(T) = \eta$ or $v'(T) = \delta$ (or both). It
follows from (\ref{eq:mixed-4decrease'}) and the definition of the
replicator dynamics that if $0 \leq t \leq T$ then
$\frac{\dot{x}'_4}{x'_4}(t) \leq - \gamma'<0$. Therefore $x'_4(T) \leq x'_4(0)$.
By lemma \ref{lm:leq-eta}, this implies that $\max_{k \in
\{5,...,N'\}} x'_k(T) < \eta$. Therefore, $v'(T) = \delta$. Due to
(\ref{eq:mixed-vdecrease'}), this implies that $\dot{v}'(T) <0$.
Therefore, there exists a time $T_1$ with $0 < T_1 < T$ such that
$v'(T_1) > \delta$, hence a time $T_2$ with $0 < T_2 < T_1 <T$ such
$v'(T_2)=\delta$, contradicting the minimality of $T$.
\end{proof}

We now conclude: it follows from lemma \ref{lm:leq-eta-delta},
equation (\ref{eq:mixed-4decrease'}) and the definition of the
replicator dynamics that for all $t\geq 0$, $x'_4(t) \leq
\exp(-\gamma't)x_4'(0)$. By (\ref{eq:leq-eta}) this implies that for
every $k$ in $\{5,...,N'\}$,
$$\forall t \geq 0, x'_k(t) < \eta \exp(-p^k_4 \gamma' t).$$
Therefore, if $p^k_4>0$ then $x'_k(t) \to 0$ as $t \to +\infty$.
\end{proof}
\end{appendix}

\end{document}